\documentclass[10pt]{amsart}

\usepackage{amsmath, amsthm, amscd, amssymb, amsfonts, pstricks}

\input amssym.def
\input amssym.tex
\newtheorem{thm}{Theorem}[section]
\newtheorem*{thm*}{Theorem}
\newtheorem{cor}[thm]{Corollary}
\newtheorem*{cor*}{Corollary}
\newtheorem{lem}[thm]{Lemma}

\newtheorem{con}{Conjecture}
\newtheorem*{con*}{Conjecture}
\newtheorem{que}{Question}
\newtheorem*{que*}{Question}
\newtheorem*{prob*}{Problem}

\theoremstyle{definition}

\theoremstyle{remark}

\newcommand{\hol}{{\text{Hol}}}
\newcommand{\aut}{{\text{Aut}}}
\newcommand{\bbZ}{{\mathbb{Z}}}
\newcommand{\cF}{{\mathcal{F}}}

\begin{document}

\title[On the linearity of the holomorph group of a free group]
{On the linearity of the holomorph group of a free group on two
generators}

\author[F. R. Cohen]{F. R. Cohen$^{*}$}
\address{Department of Mathematics,
University of Rochester, Rochester, NY 14225, U.S.A.}
\email{cohf@math.rochester.edu}
\thanks{$^{*}$Partially supported by DARPA grant number
2006-06918-01. }

\author{Vassilis Metaftsis}
\address{Department of Mathematics,
University of the Aegean, Gr-83200 Karlovassi, Samos, Greece}
\email{vmet@aegean.gr}

\author{Stratos Prassidis}
\address{Department of Mathematics and Statistics, Canisius College,
Buffalo, New York 14208, U.S.A} \email{prasside@canisius.edu}

\begin{abstract}
Let $F_n$ denote the free group generated by $n$ letters. The
purpose of this article is to show that $\hol(F_2)$, the holomorph
of the free group on two generators, is linear. Consequently, any
split group extension $G=F_2\rtimes H$ for which $H$
is linear has the property that $G$ is linear. This result gives a
large linear subgroup of $\text{Aut}(F_3)$. A second application is
that the mapping class group for genus one surfaces with two
punctures is linear.
\end{abstract}
\maketitle

\section{Introduction and Preliminaries}

The purpose of this paper is to consider whether certain families of
discrete groups given by natural semi-direct products are linear.
The {\it holomorph} of a group $G$, $\hol(G)$, is the universal
split extension of $G$:
$$1 \to G \to \hol(G) \xrightarrow{p} \aut(G) \to 1,$$
where $\aut(G)$ acts on $G$ in the obvious way. Furthermore, the
symbol $H{\rtimes}G$ denotes the semi-direct product given by the
split extension
$$1 \to G \to G{\rtimes}H \xrightarrow{p} H \to 1$$ with the precise
action of $H$ on $G$ suppressed. The group $\hol(G)$ is universal in
the sense that any semi-direct product $G{\rtimes}H$ is given by the
pullback obtained from a homomorphism $H \to {\aut}(G)$.

Recall that a group $G$ is called {\it linear} if it admits a faithful, finite
dimensional representation in $Gl(m,k)$ for a field $k$ of
characteristic zero.

The main results here addresses a special case of the following
general question stated in \cite{cclp}:
\begin{que}
Let $\Gamma$ and $\pi$ be linear groups. Let
$$1 \to \Gamma \to G \to \pi \to 1$$
be a split extension. Give conditions which imply that $G$ is
linear.
\end{que}

It should be noted that the answer is not always positive with a
basic example given by Formanek and Procesi's  ``poison
group''(\cite{fp}). Many geometrically interesting groups fit into
the scheme given above. More specifically,  examples are given by
pure braid groups (and thus braid groups), McCool subgroups of
$\aut(F_n)$, certain fundamental groups of complements of hyperplane
arrangements, certain mapping class groups, just to mention a few.

Our main interest in the above problem is when the normal subgroup
$\Gamma$ is a linear group and $\pi$ is $\aut(\Gamma)$. When $G = F_n$, $n>2$, the free group on $n$
generators, then $\hol(F_n)$ is not linear because it contains
$\aut(F_n)$ which is not linear (\cite{fp}). The main result in this
paper is that $\hol(F_2)$ is linear.

\begin{thm*}[Main Theorem]\label{thm:main}
The group $\hol(F_2)$ is linear.
\end{thm*}

The method of proof is to show that $\hol(F_2)$ contains a finite
index subgroup $\pi$  which is linear. That is done by exhibiting
explicit maps from $\pi$  to a product of two groups such that (i)
each of the two groups is linear as subgroups of ${\aut}(F_2)$ and
(ii) the product map is an embedding, thus showing that $\pi$ is
linear. Since linearity is preserved under finite extensions,
$\hol(F_2)$ is linear. In addition, the main Theorem has the
following consequence.

\begin{cor}\label{cor-extra}
Let $\pi$ be a linear group. Then the semidirect product
$F_2{\rtimes}\pi$ is linear. In particular, any group extension $G$
given by $$1 \to F_2 \to G \to \mathbb Z\to 1$$ is linear.
\end{cor}

A second corollary implies that certain mapping class groups are
linear. Let $T$ denote the torus $S^1 \times S^1$. Let $\Gamma_1^k$
denote that mapping class group for genus one surfaces with $k$
punctured points. The linearity of the case of $\Gamma_1^1$ is given by
the fact that this group is $SL(2,\mathbb Z)$.

\begin{cor}\label{cor:gamma.1.2.is.linear}
The group $\Gamma_1^2$ is linear.
\end{cor}

Moreover, in \cite{cw}, it was shown that $\hol(F_2)$ is a subgroup
of $\aut(F_3)$. Thus the Main Theorem implies that $\hol(F_2)$ is a
large, natural, linear subgroup of $\aut(F_3)$.  That is a partial answer to
a more general question.

\begin{que}\label{cs}
Find large linear subgroups of $\aut(F_n)$, $n \ge 3$.
\end{que}

The methods here do not generalize for $n > 3$ as $Aut(F_3)$, and
thus $Aut(F_n)$, $n > 3$, are not linear by Formanek and Procesi
\cite{fp}. Nonetheless, they suggest a possible extension for $n = 4$.

\begin{con}\label{con-hol}
The group $G$ defined by the natural extension
$$1 \to F_3 \to G \to \hol(F_2) \to 1.$$
with the natural action of  $\hol(F_2)$ on $F_3$, is linear.
\end{con}

A positive answer to Conjecture \ref{con-hol} has an interesting
consequence. Let $M_n$ be the McCool subgroup of $\aut(F_n)$ i.e.
the subgroup generated by basis-conjugation automorphisms
\cite{mccool}. Let $M_n^+$ be the upper-triangular McCool subgroup
as defined in section \ref{sec:proof.of.main.theorem} below or in
\cite{cpvw}. An easy calculation shows that
$$M_3^+ \cong P_3 \cong F_2{\times}{\mathbb{Z}} <
\hol(F_2).$$
Also, in \cite{cpvw}, it was shown that there is a split exact sequence for all $n$:
$$1 \to F_{n-1} \to M_n^+ \to M_{n-1}^+ \to 1.$$
Combining all the above, we see that a positive answer to Conjecture \ref{con-hol} implies the
linearity of $M_4^+$.

Notice that $\hol(F_2)$ fits into a split exact sequence:
$$1 \to F_2 \to \hol(F_2) \to \aut(F_2) \to 1.$$
The linearity question of $\aut(F_2)$ was reduced to the linearity
of the braid group on four strands. In \cite{dfg}, it was shown that
$\aut(F_2) {\times}{\mathbb{Z}}$ is commensurable with $B_4$. The
linearity of $B_4$ and the other braid groups was settled
(\cite{big}, \cite{kr1}, \cite{kr2}), proving the linearity of
$\aut(F_2)$.

The following conjecture was formulated in \cite{cclp} which
addresses the linearity question of split extensions with kernel a
free group.

\begin{con}
Let $G$ be a linear group and
$$1 \to F_n \to \Gamma \to G \to 1$$
a split exact sequence with $F_n$ a finitely generated free group.
If $G$ acts trivially on the homology of $F$, then $\Gamma$ is
linear.
\end{con}

The homological condition is needed in that generality because of
the counterexample in \cite{fp}. It should be noted that the
situation in the Main Theorem is different because the action of
$\aut(F_2)$ on $F_2$ is not trivial on the homology. One consequence
of this conjecture is that $M_n$ is linear.

Corollary \ref{cor-extra} implies that $F_2{\rtimes}{\bbZ}$ is
linear. The above remarks suggest the following conjecture.

\begin{con}
Let $F$ be a free group and
 $\Gamma = F{\rtimes}{\mathbb{Z}}$. Then $\Gamma$ is linear. That
is if
$$1 \to F \to \Gamma \to \mathbb{Z} \to 1$$
is exact, then $\Gamma$
is linear.

\end{con}

\section{Proof of the main theorem}\label{sec:proof.of.main.theorem}

Let $F_n$ be the free group with basis $\{x_1, x_2, \cdots , x_n\}$.
Let ${\chi}_{k,i}$ ($1 \le i, k \le n$) denote the elements of
${\aut}(F_n)$ defined by:
$${\chi}_{k,i}(x_j) = \left\{
\begin{array}{ll}
x_j, & \text{if}\; j \not= k \\
x_i^{-1}x_kx_i, & \text{if}\; j = k.
\end{array}\right.$$
The {\it McCool subgroup} $M_n$ of $\aut(F_n)$ is the subgroup generated by ${\chi}_{k,i}$:
$$M_n = {\langle} {\chi}_{k,i}:\; k, i = 1, 2, \dots n, \; k \not= i{\rangle}.$$
The {\it upper triangular McCool} subgroup is the subgroup generated by:
$$M_n^+ = {\langle} {\chi}_{k,i}:\; k, i = 1, 2, \dots n, \; k < i{\rangle}.$$
For the main properties of $M_n$ and $M_n^+$ see \cite{cpvw}.
There is a natural map $\aut(F_n) \to GL(n, {\bbZ})$ which is an
epimorphism with kernel denoted $IA_n$. It is known that $IA_2$ is a
free group with two generators given by ${\chi}_{2,1}$ and
${\chi}_{1,2}$. Thus there is a (non-split) group extension
$$1 \to F_2 \to \aut(F_2) \to GL(2, {\bbZ}) \to 1.$$
The fact that the kernel is $F_2$ was shown in \cite{chang} and
\cite{nielsen}. Here $F_2$ is identified with the subgroup of inner
automorphisms of $\aut(F_2)$. We write $F_2 = {\langle}{\tau}_a,
{\tau}_b{\rangle}$, the inner automorphisms of $F_2$.

A basis for $F_2$ will occur in three distinct ways below. Thus a
free group with basis $\{\alpha,\beta\}$ will be named $F_2 =
{\langle}\alpha, \beta {\rangle}$. (Thus to alert the reader, there
are 3 distinct choices of bases for $F_2$ given below by
$\{x_1,x_2\}$,  $\{a,b\}$, and $\{{\tau}_a, {\tau}_b\}$.)

The natural map induced by the mod-2 reduction  is an
epimorphism that induces an exact sequence:
$$1 \to {\Gamma}(2, 2) \to SL(2, {\bbZ}) \xrightarrow{r} SL(2, {\bbZ}/2{\bbZ}) \to 1.$$
The group ${\Gamma}(2, 2)$ is a free group of rank $2$. Also, we consider
the extension
$$1 \to {\Gamma}(2, 2) \to GL(2, {\bbZ}) \xrightarrow{(r, {\det})}
GL(2, {\bbZ}/2{\bbZ}){\times} {\bbZ}/2{\bbZ}\to 1.$$

Form the pull-back diagram:
$$\begin{CD}
1 @>>> F_2 @>>> {\cF} @>>> {\Gamma}(2,2) @>>> 1 \\
@. @VVV @VVV @VVV @. \\
1 @>>> F_2 @>>> \aut(F_2) @>>> GL(2, {\bbZ}) @>>> 1
\end{CD}$$
to obtain a morphism of extensions
$$\begin{CD}
1 @>>> F_2 @>>> {\cF} @>>> {\Gamma}(2,2) @>>> 1 \\
@. @VVV @VVV @VVV @. \\
1 @>>> F_2 @>>> \aut(F_2) @>>> GL(2, {\bbZ}) @>>> 1\\
@. @VVV @VVV @VVV @. \\
1 @>>> \{1\} @>>> GL(2, {\bbZ/ 2\bbZ}){\times}{\bbZ}/2{\bbZ} @>>>
GL(2, {\bbZ/ 2\bbZ}){\times}{\bbZ}/2{\bbZ} @>>> 1.
\end{CD}$$
This middle exact sequence is the extension
$$1 \to  IA_2 \to \aut(F_2) \to GL(2, {\bbZ}) \to 1.$$
In this case,  $IA_2$ is isomorphic to the inner automorphism group of $F_2$
generated by two elements $\chi_{i,j}$ with $i \neq j$ and $1 \leq
i,j\leq 2$.

\begin{lem}\label{lem-one}
The group $\cF$ is a subgroup of $\aut(F_2)$ of index $12$.
Furthermore, $\cF$ is generated by the inner automorphisms of $F_2$
and the automorphisms $x_i$, $i = 1, 2$,
$$\begin{array}{ll}
x_1(a) = ab^2, & x_1(b) = b \\
x_2(a) = a, & x_2(b) = ba^2.
\end{array}$$
\end{lem}

\begin{proof}
The result follows because the group
${\Gamma}(2, 2)$ is a free group on two
generators (\cite{frasch}, \cite{kulk})
$$A_1 = \begin{pmatrix}
1 & 2 \\
0 & 1
\end{pmatrix}, \quad
A_2 = \begin{pmatrix}
1 & 0 \\
2 & 1
\end{pmatrix}.$$
Notice that the image of $x_i$ is $A_i$, for $i = 1, 2$. Also, $\cF$
has index $12$ in $\aut(F_2)$ because $GL(2, {\bbZ}/2{\bbZ})$ has
order $6$.
\end{proof}

The following describes the structure of $\cF$.

\begin{lem}\label{lem-two}
The group $\cF$ can be written as a semi-direct product $\cF =
{\langle}{\tau}_a, {\tau}_b{\rangle}{\rtimes}{\langle}x_1, x_2
{\rangle}$ with the action of $x_i$ being exactly as the action of
${\langle}x_1, x_2{\rangle}$ on ${\langle}a, b{\rangle}$.
\end{lem}

\begin{proof}
Since the group
${\Gamma}(2,2)$ is a free group, the extension is split and thus a
semi-direct product. Furthermore, the extension is classified by the
map $P{\Gamma}(2,2) \to {\aut}(F_2)$ which sends $A_1$ to $x_1$ and
$A_2$ to $x_2$. The proof of the Lemma follows by inspection.
\end{proof}

Remember that there is a split exact sequence
$$1 \to F_2 \to \hol(F_2) \xrightarrow{p} \aut(F_2) \to 1.$$
Also, ${\cF} < \aut(F_2)$ and thus it is linear.
Lemma \ref{lem-one} implies that the group $\pi = p^{-1}({\cF})$ is
of index $6$ in $\hol(F_2)$ and it fits into an exact sequence
$$1 \to F_2 = {\langle}a, b{\rangle} \to \pi \to {\cF} \to 1.$$

The next Lemma is the main tool used in the proof of the Main
Theorem.

\begin{lem}\label{lem:embedding}
There are two maps
$$f_1, f_2: \pi \to {\cF}$$ such that the product
$$f_1 \times f_2: \pi \to {\cF}\times {\cF}$$
is a monomorphism. Thus $\pi$ and $\hol(F_2)$ are linear.
\end{lem}

Notice that Lemma \ref{lem:embedding} implies that $\pi$ and thus
$Hol(F_2)$ are linear. Thus the Main Theorem follows and it suffices
to prove Lemma \ref{lem:embedding}, the subject of the next section.

\section{Proof of Lemma \ref{lem:embedding}}

Recall from Lemma \ref{lem-one} that the group $\cF$ is a subgroup
of $\aut(F_2)$ of index $6$ generated by the inner automorphisms of
$F_2$ and the automorphisms $x_i$, $i = 1, 2$,
$$\begin{array}{ll}
x_1(a) = ab^2, & x_1(b) = b \\
x_2(a) = a, & x_2(b) = ba^2.
\end{array}$$ This action means that the elements $x_i$ are acting
by conjugation. So a restatement of this action is given by
$$\begin{array}{ll}
x_1\cdot a\cdot x_1^{-1} = x_1(a) = ab^2, & x_1\cdot b\cdot x_1^{-1} = x_1(b) = b \\
x_2\cdot a\cdot x_2^{-1}= x_2(a) = a, & x_2\cdot b\cdot x_2^{-1} =
x_2(b) = ba^2.
\end{array}$$

Furthermore, by Lemma \ref{lem-two}, there is an extension
$$1 \to  {\langle}{\tau}_a, {\tau}_b{\rangle} \to {\cF} \to {\langle}x_1, x_2{\rangle} \to 1$$
where the action of is specified by regarding $a =  {\tau}_a$ and $b
= {\tau}_b$:
$$\begin{array}{ll}
x_1{\tau}_ax_1^{-1} = {\tau}_a{\tau}_b^2, \\
x_1{\tau}_bx_1^{-1} = {\tau}_b, \\
x_2{\tau}_ax_2^{-1} = {\tau}_a \\
x_2{\tau}_bx_2^{-1} = {\tau}_b{\tau}_a^2.
\end{array}$$

Furthermore, the group $\pi$ is a split extension
$$1 \to F_2 = {\langle}a, b{\rangle} \to \pi \to {\cF} \to 1$$
with generators for ${\cF}$ specified above.

The additional data specifying the action of ${\cF}$ on $F_2 =
{\langle}a, b{\rangle}$ is given next.

$$\begin{array}{ll}
x_1ax_1^{-1} = ab^2\\
x_2ax_2^{-1} = a \\
x_1bx_1^{-1} = b,\\
x_2bx_2^{-1} = ba^2 \\
{\tau}_aa{\tau}_a^{-1} = a,\\
{\tau}_ba{\tau}_b^{-1} = bab^{-1}, \\
{\tau}_ab{\tau}_a^{-1} = aba^{-1}, \hbox{and}\\
{\tau}_bb{\tau}_b^{-1} = b
\end{array}$$

By a direct comparison, the above gives two distinct isomorphic
copies of $\cF$ in $\pi$.

These relations are summarized as follows:
$$\begin{array}{ll}
x_1ax_1^{-1} = ab^2, & x_2ax_2^{-1} = a \\
x_1bx_1^{-1} = b, & x_2bx_2^{-1} = ba^2 \\
x_1{\tau}_ax_1^{-1} = {\tau}_a{\tau}_b^2, & x_2{\tau}_ax_2^{-1} = {\tau}_a \\
x_1{\tau}_bx_1^{-1} = {\tau}_b, & x_2{\tau}_bx_2^{-1} = {\tau}_b{\tau}_a^2 \\
{\tau}_aa{\tau}_a^{-1} = a, & {\tau}_ba{\tau}_b^{-1} = bab^{-1} \\
{\tau}_ab{\tau}_a^{-1} = aba^{-1}, & {\tau}_bb{\tau}_b^{-1} = b
\end{array}$$

Rewrite the last two pairs of relations as follows:
$$\begin{array}{ll}
a^{-1}{\tau}_aa{\tau}_a^{-1}a = a, & b^{-1}{\tau}_ba{\tau}_b^{-1}b = a \\
a^{-1}{\tau}_ab{\tau}_a^{-1}a = b, & b^{-1}{\tau}_bb{\tau}_b^{-1}b
=b
\end{array}$$

The following hold:

$$\begin{array}{ll}
{\tau}_aa= a{\tau}_a, & b^{-1}{\tau}_ba{\tau}_b^{-1}b = a \\
a^{-1}{\tau}_ab{\tau}_a^{-1}a = b, & {\tau}_bb = b{\tau}_b
\end{array}$$

\

Change of generators by setting $t_a = a^{-1}{\tau}_a$ and $t_b =
b^{-1}{\tau}_b$. Notice that the previous relations are equivalent
to
$$\begin{array}{ll}
[t_a, a] = 1, & [t_b, a] = 1\\[0ex]
[t_a, b] = 1, & [t_b, b] = 1
\end{array}$$

Thus the group $\pi$ is generated by the set $\{a, b, t_a, t_b, x_1,
x_2\}$ with relations above equivalent to the following:
$$\begin{array}{ll}
x_1ax_1^{-1} = ab^2, & x_2ax_2^{-1} = a \\
x_1bx_1^{-1} = b, & x_2bx_2^{-1} = ba^2 \\
x_1t_ax_1^{-1} = t_at_b^2, & x_2t_ax_2^{-1} = t_a \\
x_1t_bx_1^{-1} = t_b, & x_2t_bx_2^{-1} = t_bt_a^2 \\[0ex]
[t_a, a] = 1, & [t_b, a] = 1\\[0ex]
[t_a, b] = 1, & [t_b, b] = 1
\end{array}$$

Thus the group $\pi$ has a normal subgroup $N(\pi)$ generated by the
set $\{a, b, t_a, t_b\}$ with the following properties.
\begin{enumerate}
  \item The subgroup $N(\pi)$ is isomorphic to a direct product of two free
  groups ${\langle}a, b{\rangle}{\times}{\langle}t_a,
  t_b{\rangle}.$

  \item The cokernel $\pi/N(\pi)$ is isomorphic to a free group ${\langle}x_1,
  x_2{\rangle}$.
  \item There is a homomorphisms $h: \pi \to \cF$  specified by
  sending

\begin{enumerate}
  \item  $t_a$ and $t_b$ to $1$
  \item $x_i$ to $x_i$,
  \item $a$ to $a$ and $b$ to $b$.
\end{enumerate}

\item The kernel of $h$ is the free group ${\langle}t_a, t_b{\rangle}$.
\end{enumerate}

Notice that the intersection of kernels of $ker(h) \cap ker(p)$ is
the intersection of
$${\langle}t_a, t_b{\rangle}\cap {\langle}a,b{\rangle} = \{1\}.$$
Furthermore, the maps $f_1, f_2$ of Lemma \ref{lem:embedding} are
given by $f_1 = h$ and $f_2 = p.$

Therefore $\pi = ({\langle}a,
b{\rangle}{\times}{\langle}t_a,t_b{\rangle}) {\rtimes}{\langle}x_1,
x_2{\rangle}$. Then ${\langle}a, b{\rangle}$ and ${\langle}t_a,
t_b{\rangle}$ are normal subgroups of $\pi$ and thus $\pi$ admits two
epimorphisms:
$$\begin{array}{l}
f_1: \pi \to {\langle}a, b{\rangle}{\rtimes}{\langle}x_1, x_2{\rangle}
\cong
{\cF}\\
f_2: \pi \to {\langle}t_a, t_b{\rangle}{\rtimes}{\langle}x_1,
x_2{\rangle} \cong {\cF}.
\end{array}$$
But $\cF$ is linear, as a subgroup of $\aut(F_2)$ and $\text{ker}(f_1) =
{\langle}t_a, t_b{\rangle}$ and
$\text{ker}(f_2) = {\langle}a, b{\rangle}$. Since
$\text{ker}(f_1){\cap}\text{ker}(f_2) = \{1\}$, the composition
$$\pi \xrightarrow{\Delta} \pi{\times}\pi \xrightarrow{f_1{\times}f_2} {\cF}{\times}{\cF}$$
is a monomorphism, where $\Delta$ is the diagonal map. Since
${\cF}{\times}{\cF}$ is linear, $G$ is linear. But $\pi$ has index
$6$ in $\hol(F_2)$. Thus $\hol(F_2)$ is linear, completing the proof
of Lemma \ref{lem:embedding} and the Main Theorem.

\section{Proof of Corollary \ref{cor-extra}}

Let $\pi$ be a linear group. Let
$$1 \to F_2 \to G \xrightarrow{p} \pi \to 1$$
be a split extension. The result to be proven is that $G$ is linear.
The split extension induces a commutative diagram of exact
sequences:
$$\begin{CD}
1 @>>> F_2 @>>> G @>{p}>> \pi @>>> 1  \\
@. @| @VV{i}V @VV{j}V @. \\
1 @>>> F_2 @>>> \hol(F_2) @>>> \aut(F_2) @>>> 1
\end{CD}$$
where $j$ is the map induced by the action of $\pi$ on $F_2$. Notice
that the right-hand diagram is a pull-back diagram. Thus the map
$$i{\times}p: G \to \hol(F_2){\times}{\pi}$$
is an injection. Since $\pi$ and $\hol(F_2)$ are linear, $G$ is
linear.

\section{Proof of Corollary \ref{cor:gamma.1.2.is.linear}}

Let $\Pi$ denote the group of orientation preserving homeomorphisms
$Top^+(T)$, and $$\mbox{Conf}(T,k) = \{(z_1, \cdots, z_k)\in T^k|
z_i \neq z_j \ \mbox{if} \ i \neq j\}$$ the configuration space of
$k$ points in $T$. Write $\text{Top}^+(T, Q_k)$ for the topological
group of the orientation preserving self-homeomorphisms of $T$ that
leave $Q_k$, a set of $k$ distinct points in $T$, invariant.
Similarly, we write $P\text{Top}^+(T, Q_k)$ for the
orientation-preserving homeomorphisms of $T$ that fix $Q_k$
pointwise. Denote
$${\Gamma}_1^{k} = {\pi}_0(\text{Top}^+(T, Q_k))\ \  {\rm and}\ \  P{\Gamma}_1^{k} ={\pi}_0(P\text{Top}^+(T, Q_k)),$$ for the
corresponding mapping class groups.

Recall the following facts \cite{cohen}.
\begin{enumerate}
  \item If  $k \geq 2 $, then the spaces
$$E\Pi\times_{\Pi} \mbox{Conf}(T,k),$$ and $$E\Pi\times_{\Pi}
\mbox{Conf}(T,k)/\Sigma_k$$ are respectively $K( P\Gamma_1^k,1 )$,
and $K(\Gamma_1^k,1 )$.
  \item Furthermore, $E\Pi\times_{\Pi} \mbox{Conf}(T,k)$ is homotopy
equivalent to $$ESL(2,\mathbb Z) \times_{SL(2,\mathbb Z)}
\mbox{Conf}(T -Q_1,k-1)$$ where $Q_1=\{ (1,1)\}\subset T$. Thus there is a fibration
$$ESL(2,\mathbb Z)\times_{SL(2,\mathbb Z)} \mbox{Conf}(T,2) \to BSL(2,\mathbb Z)$$ with
fibre $T -Q_1.$

\end{enumerate}

Using the above one can easily see that the group $P\Gamma_1^1$ is isomorphic to $SL(2,\mathbb Z)$.
Also, the kernel of the natural mod-$2$ reduction map
$$SL(2,\mathbb Z) \to SL(2,\mathbb Z/ 2\mathbb Z)$$ denoted
$S\Gamma(2,2)$ here is a free group on two letters. So, since the
fundamental group of $T^2 -Q_1$ is free on two letters, the
fundamental group of
$$\pi_1(ESL(2,\mathbb Z)\times_{SL(2,\mathbb Z)} {\rm Conf}(T,2)) = P\Gamma_1^2$$ has an
index six subgroup $K$ which admits an extension $$1 \to F_2 \to K
\to \Gamma(2,2) \to 1$$ and is split. Therefore, by the main Theorem \ref{thm:main},
the group $K$ is linear and thus $P\Gamma_1^2$ is linear. Notice that $P\Gamma_1^2$ has index
two in $\Gamma_1^2$ and therefore $\Gamma_1^2$ is linear.

\section{On large linear subgroups of $Aut(F_n)$}

In this final section we present a small step towards understanding Question \ref{cs}.
For any group $G$, there is a group homomorphism defined $$E:
{\hol}(G) \to {\aut}(G*F)$$  and shown to be a monomorphism where
$F$ is a free group \cite{cw}. Explicitly, the homomorphism is
defined as follows:
\begin{itemize}
\item For $f\in {\aut}(G)$,
$$E(f)(z) = \left\{
\begin{array}{ll}
f(z), & \text{if}\; z\in G \\
z, & \text{if}\; z\in F.
\end{array}
\right.$$
\item For $h\in G$,
$$E(h)(z) = \left\{
\begin{array}{ll}
z, & \text{if}\; z\in G \\
hzh^{-1}, & \text{if}\; z\in F.
\end{array}\right.$$
\end{itemize}

It is known that $\aut(F_3)$ is not linear (\cite{fp}). But by the
above, it obvious that $\hol(F_2)$ is a subgroup of $\aut(F_3)$.
Thus $\aut(F_3)$ is not linear but contains a large, `natural',
linear subgroup.

\end{document}